\documentclass[12pt,a4paper]{article}
\usepackage{mathrsfs}
\usepackage{amsfonts}
\usepackage{amsmath}
\usepackage{cite}
\numberwithin{equation}{section}
\usepackage{amsthm}
\newtheorem{lemma}{Lemma}[section]
\newtheorem{corollary}[lemma]{Corollary}
\newtheorem{definition}{Definition}[section]
\newtheorem{theorem}[definition]{Theorem}
\newtheorem{remark}{Remark}[section]
\begin{document}
\title{\bf Global solutions to planar  magnetohydrodynamic equations with radiation and large initial data }
\author{Xulong Qin$^{1,2}$
\thanks{\footnotesize{Corresponding author
\newline\indent~~E-mail: ~qin\_xulong@163.com,~~ mcsyao@mail.sysu.edu.cn
\newline\indent~~2000 \it Mathematics Subject
Classification.} ~35B45; 35L65; 35Q60; 76N10;76W05.\newline
   }\qquad Zheng-an Yao$^{1}$\\
\it \small $^1$Department of Mathematics, Sun Yat-sen University,\\
 \it\small Guangzhou 510275, People's Republic of China\\
\it \small $^2$The Institute of Mathematical Sciences, Chinese University of HongKong,\\
\it \small Shatin,  HongKong}
\date{}
 \maketitle
 \begin{abstract}
A global existence result is established for a free boundary problem of
planar   magnetohydrodynamic fluid flows with radiation and large initial data.
 Particularly, it is novelty to embrace the constant transport coefficient.
As a by-product, the free boundary is shown to expand outward at an algebraic rate from above in time.
 \end{abstract}
 \section{Introduction}
In astrophysics, stars may be viewed as compressible fluid flows
 formulated by the Navier-Stokes equations, which is now expressed
 by the conservation of mass, balance of momentum and energy.
 However, their dynamics are often influenced by  magnetic fields and high temperature radiation.
 When the radiation is taken into account, the new balance law for
 the intensity of radiation should be added to complete the
 hydrodynamic system. More precisely, the material flow needs a
 relativistic treatment since the photons are massless particles
 traveling at the speed of light. In other words, it raises more
 difficulties in mathematics and physics. Fortunately, under some
 plausibly physical hypotheses and asymptotic analysis, especially for
 the equilibrium diffusion model,
 the radiative transfer equation can be replaced by the usual hydrodynamic model
 equations with the pressure, the internal energy and the heat
 conductivity added by the special radiative components, see for instance \cite{Mihalas}.
Furthermore, when the magnetic fields are considered, the
 motions of conducting fluids may induce electric fields. Thus, the
 complex interaction of magnetic and electric fields significantly affects
 the hydrodynamic motion of fluid flows, which is described by Maxwell's equation. Finally,  the system of
 magnetohydrodynamics
  is the following in Eulerian coordinates.
\begin{subequations}\label{qian2}
\begin{align}
\partial_{t}\rho+\text{div}(\rho \mathbf{u})&=0, \label{qian21}\\[3mm]
\partial_{t}(\rho \mathbf{u})+\text{div}(\rho
\mathbf{u}\otimes\mathbf{u})+\nabla p&=\text
{div} \mathbb{S}+(\nabla\times\mathbf{H})\times\mathbf{H},\label{qian22}\\[3mm]
\partial_t(\rho \mathcal{E})+\text{div}(\rho \mathcal{E'}
\mathbf{u}+p\mathbf{u})+\text{div}\mathbf{Q}&=
\text{div}(\mathbb{S}\textbf{u}+\nu\mathbf{H}\times(\nabla\times\mathbf{H}))\notag\\[2mm]
&+\text{div}((\mathbf{u}\times\mathbf{H})\times\mathbf{H}),\label{qian23}\\[3mm]
\mathbf{H}_t-\nabla\times(\mathbf{u}\times\mathbf{H})&
=-\nabla\times(\nu\nabla\times\mathbf{H}),\quad
\text{div}\mathbf{H}=0,\label{qian4}
\end{align}
\end{subequations}
where $\rho$ is the density of fluid flows, $p$ is the pressure,
$\mathbf{u}\in \mathbb{R}^3$ the velocity, $\mathbf{H}\in
\mathbb{R}^3$ the magnetic field. $\mathcal{E}$ stands for the total
energy expressed by
\begin{equation*}
\mathcal{E}=\rho\left(e+\frac12|\mathbf{u}|^2\right)+\frac12|\mathbf{H}|^2,\quad
\mathcal{E'}=\rho\left(e+\frac12|\mathbf{u}|^2\right),
\end{equation*}
which $e$ is the internal energy of flows. The viscous stress tensor
\begin{equation*}
\mathbb{S}=\lambda'(\text{div}\mathbf{u})\mathbb{I}+\mu(\nabla\mathbf{u}+\nabla\mathbf{u}^{\top}),
\end{equation*}
where $\lambda'$ and $\mu$ are the viscosity coefficients of the
flows with $\lambda'+2\mu>0$. $\mathbb{I}$ is the $3\times 3$
identity matrix. $\nabla\mathbf{u}^{\top}$ is the transpose of the
matrix $\nabla\mathbf{u}$.  Particularly, we emphasis that the
electric field is induced by the velocity $\textbf{u}$ and the
magnetic field $\mathbf{H}$,
\begin{equation*}
\mathbf{E}=\nu\nabla\times\mathbf{H}-\mathbf{u}\times\mathbf{H},
\end{equation*}
and the induction equation \eqref{qian4} is derived from
neglecting the displacement current in Maxwell's equation. For
more detailed physical explanation and mathematical deduction, see
the appendix of \cite{ChenWang02} for reference.

The pressure $p$ of the gas obeys the equation of state
\begin{equation}\label{p}
p=p_{G}+p_{R}=R\rho\theta+\frac{a}{3}\theta^4.
\end{equation}
$R$ is a constant depending on the material properties of the gas,
$\theta$ the temperature of fluid flows and the last term in the
above equality denotes the radiative pressure with $a>0$ being the
stefan-Boltzmann constant. Accordingly, the internal energy is
\begin{equation}\label{e}
e=e_{G}+e_{R}=C_v\theta+\frac{a}{\rho}\theta^4
\end{equation}
with $C_v$ being the heat capacity of the gas at constant volume,
and similarly the heat flux $\mathbf{Q}$ is
\begin{equation*}
\mathbf{Q}=-\kappa\nabla\theta=-\left(\kappa_1+\frac{4ac\theta^3}{\hat{\kappa}\rho}\right)\nabla\theta,
\end{equation*}
where $\kappa_1$ is a positive constant and $c$ is the speed of
light. The quantity $1/\hat{\kappa}\rho$ is the mean free path of a
photon inside medium, which is related to $\theta$. Without loss of generality, we can assume that
\begin{equation}
 \kappa_1(1+\theta^q)\leq \kappa= \kappa(\rho,\theta) \leq \kappa_2(1+\theta^q),\,\,q\geq 0,\label{heatconductivity}
\end{equation}
where $\kappa_1,\kappa_2$ are positive constants.

In the present paper, we primarily study a free boundary problem for
planar magnetohydrodynamic fluid flows. More precisely, the motion
of flows is supposed to be in the $x$-direction and uniform in the
transverse direction $(y,z)$, i.e.,
\begin{equation*}
\begin{split}
&\rho=\rho(x,t),\qquad \theta=\theta(x,t),\\[3mm]
&\mathbf{u}=(u,\mathbf{w})(x,t),\qquad \mathbf{w}=(u_2,u_3),\\[3mm]
&\mathbf{H}=(b_1,\mathbf{b})(x,t),\qquad \mathbf{b}=(b_2, b_3),
\end{split}
\end{equation*}
and the corresponding dynamic equations \eqref{qian2} are reduced
to the following in the Eulerian coordinates.
\begin{subequations}\label{subequation1}
 \begin{align}
 &\rho_t+(\rho u)_x=0,\qquad x\in \Omega_t:=(a(t),b(t)), \,\,t>0,\label{sub1}\\[2mm]
 &(\rho u)_t+(\rho u^2+p+\frac12 |\mathbf{b}|^2)_x
 =(\lambda u_x)_x,\label{sub2}\\[2mm]
 &(\rho \mathbf{w})_t+(\rho u \mathbf{w}-\mathbf{b})_x=(\mu
 \mathbf{w}_x)_x,\label{sub3}\\[2mm]
 &\mathbf{b}_t+(u\mathbf{b}-\mathbf{w})_x=(\nu
 \mathbf{b}_x)_x,\label{sub4}\\[2mm]
& (\rho e)_t+(\rho e u)_x-(\kappa\theta_x)_x=\lambda u_x^2+\mu
|\mathbf{w}_x|^2+\nu |\mathbf{b}_x|^2-pu_x. \label{sub5}
 \end{align}
 \end{subequations}
where $\lambda=\lambda'+2\mu$ and $b_1=1$.

To supplement the system \eqref{subequation1}, we impose the
following initial conditions,
\begin{equation}
(\rho,u,\theta, \mathbf{w},\mathbf{b})(x,0)=(\rho_0,u_0,\theta_0,
\mathbf{w}_0,\mathbf{b}_0)(x).
\end{equation}
 and the free boundary conditions
\begin{equation}
(\lambda u_x-p,\mathbf{w},\mathbf{b},\theta_x)(f(t),t)=0, \qquad
f(t)=a(t),\,\,b(t). \label{freefreeboundary1}
\end{equation}
where $f(t)$ denotes the free boundary defined by $f'(t)=u(f(t),t)$.

The aim of this paper is to establish the well-posedness of the system \eqref{subequation1}--\eqref{freefreeboundary1}
for the general setting of the heat conductivity \eqref{heatconductivity},
especially including the case of constant transport coefficients.

Let us first review some related works in this line.
For instance, Zhang-Xie \cite{ZhangXie} studied the global existence of the equations \eqref{subequation1} for
Dirichlet boundary problem when $q>\frac{5}{2}$, which is extended to $q>(2+\sqrt{211})/9$ by Qin-Hu \cite{QinY}. In the forthcoming companion paper \cite{QINYAO2}, we also showed that global solutions exist for large initial data even if $q\geq 0$.
On the other hand,
Chen-Wang \cite{ChenWang02,Wang} established the existence of
global solutions of real gas for a free boundary or Dirichlet
problem, respectively, where they assume that the pressure and internal energy satisfy the following condition with exponent $r\in [0,1]$ and $q\geq 2(r+1)$.
\begin{equation*}
0\leq \rho p\leq p_0(1+\theta^{r+1}), \qquad e_\theta\geq e_0(1+\theta^r).
\end{equation*}
We remark that the  restriction on $q$ in \eqref{heatconductivity} is
only from a more mathematical point of view. The main difficulties come from the interaction
 of magnetic fields and fluid velocity, as well as higher nonlinearity of radiative term, which prevent solving the initial
 boundary problem by  means of known analytic techniques and tools.
 As a consequence, some essential new ideas are proposed to track this open problem.

From the previous works, we know that one of  the key points of the problem is  to achieve the upper and lower bounds of density and temperature,
which implies that there is no concentration of mass and heat.
By delicate analysis, we notice that the radiative term helps us get the upper and lower bounds of density based on entropy-type energy estimates.
Due to the interaction of the dynamic motion of the fluids and magnetic field, a priori estimates of temperature are much more complex.
In our context, new a priori estimates of $||\theta^8||_{L^1[0,1]}$ and $||\rho_x||_{L^2[0,1]}$, which  are controlled by
$||\mathbf{w}_{xx}||_{L^2([0,1]\times[0,t])}$ and $||\mathbf{b}_{xx}||_{L^2([0,1]\times[0,t])}$, are proposed, see Lemma \ref{lemma09} and Lemma \ref{lemma10},
and then a priori estimates of $||\mathbf{w}_{xx}||_{L^2([0,1]\times[0,t])}$ and $||\mathbf{b}_{xx}||_{L^2([0,1]\times[0,t])}$ are subsequently obtained. With these bounds, all required a priori estimates for  $||\theta^8||_{L^1[0,1]}$, $||\rho_x||_{L^2[0,1]}$ and $||u_x||_{L^4([0,1]\times[0,t])}$ are achieved, refer to \eqref{density}--\eqref{velocity}. In the subsequential process, motivated by \cite{ChenWang02,jiangJDE94, kawohl}, we succeed in getting the
upper bound of temperature and the first derivative of velocity.

To our best knowledge, there is a few results for the case of perfect flows, namely neglecting the radiative effect, or equivalently $a=0$ in \eqref{p}, when the transport coefficients are positive constants.
Among them, Kawashima and Okada \cite{Kawashima}
proved the existence of global smooth solutions to one-dimensional
motion  with small initial data. In addition, Hoff and Tsyganov \cite{Hoff} considered
uniqueness and continuous dependence on initial data of weak
solution of the equations of compressible magnetohydrodynamics.
Moreover, Fan-Jiang-Nakanura \cite{FanJiang} showed the global
weak solutions of plane magnetohydrodynamic compressible flows
converge to a solution of the original equations with zero shear
viscosity as the shear viscosity goes to zero when $q\geq 1$.
However, it is an open problem for the global existence of perfect flows with large initial date, which may be shed light on by our techniques.
From our analysis, we know that the focus of the problem is on the lower bound of density since the lower bound of density in our context is strongly depended on the radiative term.

However, similar obstacle on $q$  is  also in existence even if we neglect the magnetic fields.
For instance,  Ducomet-Zlotnik \cite{Ducomet&Zlotnik} proved the existence and asymptotic
behavior for 1D radiative and reactive gas when $q\geq 2$ and
Umehara-Tani \cite{umeharaTani,umeharaTaniAA} made further extension
in this direction for 1D or spherically symmetric case when $3\leq
q<9$, which extended to $q\geq 0$ by authors in \cite{QINYAO1} . In addition, Wang and Xie \cite{WangXie} showed global
existence of strong solutions for the Cauchy problem when
$q>\frac52$ and the reference \cite{DucometZlotnikhigher} proved
global-in-time bounds of solutions and established it global
exponential decay in the Lebegue and Sobolev spaces  when $q\geq 2$.

Indeed, there are also extensively studies on MHD
hydrodynamics, which is beyond our ability to address exhaustive references, see for instance
\cite{DFeireisl05,Ducomet06,HuWangCMP,HuWangJDE,Rohde,RohdeYong,Secchi,Wangproceeding03,wang1,Zhongjiang} and references
cited therein.

We formulate our problem and state the main results in section 2, and In section 3, we
give the essential a priori estimates for the existence of global solutions.
\section{Preliminaries and Main Result}
Before stating the main result, we first introduce the Lagrangian
coordinates to translate the free  boundary problem
\eqref{subequation1}--\eqref{freefreeboundary1} to the fixed one for convience, which
is equivalent under consideration. Let
\begin{equation}\label{transformation}
y=\int_{a(t)}^x\rho(\xi,t)d\xi,\qquad  t=t.
\end{equation}
Then $0\leq y\leq 1=\int_{0}^{1}\rho(\xi,t)d\xi$ which is
 the total mass of fluid flows without loss of generality.
The system \eqref{subequation1} canonically becomes
\begin{subequations}\label{subequation}
 \begin{align}
 v_t&=u_y,\label{sub11}\\[2mm]
 u_t&=\left(-p-\frac12 |\mathbf{b}|^2+\frac{\lambda
 u_y}{v}\right)_y,\label{sub22}\\[2mm]
 \mathbf{w}_t&=\left(\mathbf{b}+\frac{\mu
 \mathbf{w}_y}{v}\right)_y,\label{sub33}\\[2mm]
 (v\mathbf{b})_t&=\left(\mathbf{w}+\frac{\nu
 \mathbf{b}_y}{v}\right)_y,\label{sub44}\\[2mm]
 e_t&=\left(\frac{\kappa}{v}\theta_y\right)_y
 +\left(-p+\frac{\lambda}{v}u_y\right)u_y+\frac{\mu |\mathbf{w}_y|^2}{v}+\frac{\nu |\mathbf{b}_y|^2}{v}, \label{sub55}
 \end{align}
 \end{subequations}
 where $v=1/\rho$ is the specific volume. The corresponding initial data and boundary condition are as follows
\begin{equation}
(v,u,\theta,\mathbf{w},\mathbf{b})(y,0)=(v_0(y),
u_0(y),\theta_0(y),\mathbf{w}_0(y),\mathbf{b}_0(y)).
\end{equation}
and
\begin{equation}
(\lambda
\frac{u_y}{v}-p,\mathbf{w},\mathbf{b},\theta_y)|(d,t)=0,\qquad
d=0,1. \label{boundary}
\end{equation}

With these preliminaries, we are now ready to state the main
result for the system \eqref{subequation}--\eqref{boundary}.
\begin{theorem}\label{thm}
Let $q\geq 0$ and $\alpha \in (0,1)$. Suppose that $\lambda,
\mu,\nu$ are positive constants. In addition, assume that the
initial data $v_0(y)$, $u_0(y)$,$\mathbf{w}_0(y)$,
$\mathbf{b}_0(y)$, $\theta_0(y)$ satisfy
\begin{equation*}
C_0^{-1}\leq v_0(y),\quad \theta_0(y)\leq C_0,
\end{equation*}
for some positive constant $C_0$ and
\begin{equation*}
(v_0(y),u_0(y),\mathbf{w}_0(y),\mathbf{b}_0(y),\theta_0(y))\in
C^{1+\alpha}(\Omega)\times C^{2+\alpha}(\Omega)^6,
\end{equation*}
for $\alpha\in (0,1)$.  Then there exists a unique solution
$(v,u,\mathbf{w},\mathbf{b},\theta)$ of the initial boundary value
problem \eqref{subequation}--\eqref{boundary} such that
\begin{equation*}
(v,v_y, v_t)\in C^{\alpha,\alpha/2}(Q_T)^3,
\end{equation*}
and
\begin{equation*}
 (u,\mathbf{w},\mathbf{b},\theta)\in
C^{2+\alpha,1+2/\alpha}(Q_T)^6.
\end{equation*}
Moreover, the expand rate of interface is
\begin{equation*}
0<b(t)-a(t)\leq C(1+t),
\end{equation*}
where $C$ is a positive constant independent of time $t$.
\end{theorem}
\begin{remark}
Obviously, the case of constant transport coefficients is included when $q=0$. In addition, it is also valid
for $\kappa=\kappa_1+\kappa_2\frac{\theta^q}{\rho}$ instead of \eqref{heatconductivity}.
\end{remark}

The procedure for the existence of global-in-time solutions is classical from the standard method by the global a priori estimates of $(\rho, u, \theta,
\mathbf{w},\mathbf{b})$. Therefore, the main task for us is to establish the global a priori estimates.
\section{Some a priori estimates}
In order to establish the existence of global solutions, we need
to deduce some a priori estimates for $(\rho, u, \theta,
\mathbf{w},\mathbf{b})$, which is essential to extending the local
solutions by fixed point theorem. In the sequel, the capital
$C(C(T))$ will denote generic positive constant depending on the
initial data (time $T>0$), which may be different from line to
line.

\subsection{Upper and lower bounds of density}

As usual, the basic energy estimates is as follows.
\begin{lemma}\label{Lemma1}
We have
\begin{equation}
\int_0^1\left(e+\frac12(u^2+|\mathbf{w}|^2+v|\mathbf{b}|^2)\right)dy\leq
C. \label{lemma1}
\end{equation}
\end{lemma}
\begin{proof}
From \eqref{sub11}--\eqref{sub55}, we get from the
boundary conditions \eqref{boundary}
\begin{equation*}
\frac{d}{dt}\int_0^1\left(e+\frac12(u^2+|\mathbf{w}|^2+v|\mathbf{b}|^2)\right)dy=0,
\end{equation*}
which implies \eqref{lemma1}.
\end{proof}
And then, we immediately arrive at the uniformed upper boundedness of density for all time.
\begin{lemma}\label{Lemma4}
Suppose the hypotheses of Theorem \ref{thm} are valid, then
\begin{equation}
\rho(y,t)\leq C,\label{rho}
\end{equation}
and
\begin{equation}
\int_0^1\theta^4dy\leq C,\label{coro1}
\end{equation}
and
\begin{equation}
\int_0^1|\mathbf{b}|^2dy\leq C.\label{magnetic}
\end{equation}
for $(y,t)\in (0,1)\times (0,T)$.
\end{lemma}
\begin{proof}
The equation \eqref{sub22} can be rewritten as
\begin{equation*}
u_t=\left(\lambda (\ln v)_t-p-\frac12|\mathbf{b}|^2\right)_y.
\end{equation*}
Integrating it over $[0,y]\times [0,t]$, we obtain
\begin{equation}
\begin{split}
\lambda \ln v=\lambda \ln v_0(y)+\int_0^t(p+\frac12|\mathbf{b}|^2)ds
+\int_0^y(u-u_0)dx,\label{lemma2V}
\end{split}
\end{equation} and notice that
\begin{equation*}
\left|\int_0^y(u-u_0)dx\right|\leq
C\int_0^1u^2dy+C\int_0^1u_0^2dy\leq C,
\end{equation*}
which implies
\begin{equation*}
\ln v\geq -C,
\end{equation*}
or equivalently
\begin{equation*}
\rho(y,t)\leq C.
\end{equation*}
Following from  \eqref{lemma1} and \eqref{rho}, we easily get
\begin{equation*}
\int_0^1\theta^4dy=\int_0^1\rho v\theta^4dy\leq
C\int_0^1v\theta^4dy\leq C\int_0^1edy\leq C.
\end{equation*}
Similarly, one has
\begin{equation*}
\int_0^1|\mathbf{b}|^2dy=\int_0^1\rho v|\mathbf{b}|^2dy\leq
C\int_0^1v|\mathbf{b}|^2dy\leq C.
\end{equation*}
This ends the proof.
\end{proof}
In the sequel, the expand rate  of free boundaries is obtained.
\begin{lemma}It holds that
\begin{equation}
0<\int_0^1vdy\leq C(1+t), \label{freelemma33}
\end{equation}
or equivalently, by the transformation \eqref{transformation}, one has
\begin{equation*}
0<b(t)-a(t)\leq C(1+t).
\end{equation*}
\end{lemma}
\begin{proof}
Integrating $[0,y]$ over \eqref{sub22}, it yields
\begin{equation*}
\int_0^yu_tdx=\frac{\lambda
u_y}{v}-(p+\frac12|\mathbf{b}|^2),
\end{equation*}
which implies by multiplying  $v$ and integrating with respect to
$y$ and $t$
\begin{equation*}
\begin{split}
&\lambda\int_0^1v dy\\
&=\lambda\int_0^1v_0(y)
dy+\int_0^t\int_0^1v\left(p+\frac12|\mathbf{b}|^2\right)dyds\\
&+\int_0^t\int_0^1v\left(\int_0^yu_tdx\right)dyds.
\end{split}
\end{equation*}
For the second term on the right hand side,  it is deduced that
\begin{equation*}
\int_0^t\int_0^1v\left(p+\frac12|\mathbf{b}|^2\right)dyds\leq
C\int_0^t\int_0^1\left(e+v|\mathbf{b}|^2\right)dyds\leq C(1+t).
\end{equation*}
On the other hand, we get from \eqref{sub11} by integration
\begin{equation*}
v(y,t)=v_0(y)+\int_0^tu_y(y,s)ds.\label{}
\end{equation*}
Thus, it leads the third term on the right hand side to
\begin{equation*}
\begin{split}
&\int_0^t\int_0^1v\left(\int_0^yu_tdx\right)dyds\\
&=\int_0^1\left(v_0(y)+\int_0^tu_y(y,s)ds\right)\left(\int_0^yudx\right)dy\\
&+\int_0^t\int_0^1u^2dyds-\int_0^1v_0(y)\left(\int_0^yu_0(x)dx\right)dy\\
&=\int_0^1v_0(y)\left(\int_0^y(u-u_0(x))dx\right)dy+\int_0^t\int_0^1u^2dyds\\
&-\int_0^1u(y,t)\left(\int_0^tu(y,s)ds\right)dy, \\
&\leq C(1+t).
\end{split}
\end{equation*}
Thus, we utilize \eqref{lemma1} to finish the proof.
\end{proof}

\begin{lemma}
One has
\begin{equation}
U(t)+\int_0^tV(\tau)d\tau\leq C(T), \label{lemma3}
\end{equation}
where
\begin{equation}
U(t)=\int_0^1\left(C_v(\theta-1-\log \theta)+R(v-1-\log
v)\right)dy,\notag
\end{equation}
and
\begin{equation}
V(t)=\int_0^1\left(\frac{\lambda u_y^2}{v\theta}+\frac{\mu
|\mathbf{w}_y|^2}{v\theta}+\frac{\nu
|\mathbf{b}_y|^2}{v\theta}+\frac{\kappa\theta_y^2}{v\theta^2}\right)dy,\notag
\end{equation}
for $0\leq t\leq T$.
\end{lemma}
\begin{proof}
According to the expression of $p$ and $e$, we thereby compute
\begin{equation*}
e_{\theta}\theta_t+\theta p_{\theta}u_y=\frac{\lambda
u_y^2}{v}+\frac{\mu |\mathbf{w}_y|^2}{v}+\frac{\nu
|\mathbf{b}_y|^2}{v}+\left(\frac{\kappa}{v}\theta_y\right)_y,
\end{equation*}
which implies that
\begin{equation*}
\begin{split}
&\frac{d}{dt}\int_0^1\left(C_v \log \theta+R \log
v+\frac43av\theta^3\right)dy\\
&=\int_0^1\left(\frac{\lambda
u_y^2}{v\theta}+\frac{\mu |\mathbf{w}_y|^2}{v\theta}+\frac{\nu
|\mathbf{b}_y|^2}{v\theta}+
\frac{\kappa\theta_y^2}{v\theta^2}\right)dy.
\end{split}
\end{equation*}
Integrating that over $(0,1)\times (0,t)$, it yields
\begin{equation*}
U(t)+\int_0^tV(\tau)d\tau\leq C\left(1+\int_0^1v\theta^3dy\right).
\end{equation*}
On the other hand, one has by \eqref{freelemma33}
\begin{equation*}
\begin{split}
\int_0^1v\theta^3dy&\leq \left(\int_0^1v\theta^4dy
\right)^{\frac34}\left(\int_0^1vdy\right)^{\frac14}\\[2mm]
&\leq C(T),
\end{split}
\end{equation*}
which ends the proof in conjunction with \eqref{lemma1}.
\end{proof}

Furthermore, we can deduce some  a priori estimates for $\theta$ and
$\mathbf{b}$ for all time.
\begin{lemma}\label{Lemma6}
We have
\begin{equation}
\int_0^t\max_{[0,1]}\theta^{q+4}(y,s)ds\leq C(T),\qquad  \quad q\geq 0, \label{lemma61}
\end{equation}
and
\begin{equation}
\int_0^t||\mathbf{b}||_{L^{\infty}([0,1])}^2ds\leq C(T),
\label{lemma62}
\end{equation}
for $0<t<T$.
\end{lemma}
\begin{proof}
There exists $y(t)$ by the mean value theorem such that
\begin{equation*}
\theta(y(t),t)=\int_0^1\theta dy,
\end{equation*}
where $y(t)\in [0,1]$ for each $t\in [0,T]$. Furthermore, it yields
by H\"{o}lder inequality
\begin{equation*}
\begin{split}
\theta(y,t)^{\frac{q+4}{2}} &=\left(\int_0^1\theta
dy\right)^{\frac{q+4}{2}}
+\frac{q+4}{2}\int_{y(t)}^y\theta(\xi,t)^{\frac{q+4}{2}-1}\theta_{\xi}(\xi,t)d\xi\\[2mm]
&\leq
C\left(1+\int_0^1\frac{\kappa^{\frac12}|\theta_y|}{v^{\frac12}\theta}
\cdot\frac{v^{\frac12}\theta^{\frac{q+4}{2}}}{\kappa^{\frac12}}dy\right)\\
&\leq
C\left[1+\left(\int_0^1\frac{v\theta^{q+4}}{\kappa}dy\right)^{\frac12}V(t)^{\frac12}\right]\\
&\leq C\left[1+\left(\int_0^1\frac{v\theta^{q+4}}{1+\theta^q}dy\right)^{\frac12}V(t)^{\frac12}\right]\\
&\leq
C\left[1+\left(\int_0^1v\theta^4dy\right)^{\frac12}V(t)^{\frac12}\right]\\
&\leq C\left(1+V(t)^{\frac12}\right).
\end{split}
\end{equation*}
Then taking square on both sides of  the above inequality, we can get \eqref{lemma61} by
integrating it over $[0,t]$ with the help of \eqref{lemma1} and
\eqref{lemma3}.

The proof of \eqref{lemma62} is directly deduced from
\eqref{lemma1} and \eqref{lemma61}. Indeed, recalling the boundary
conditions \eqref{boundary}, we find that
\begin{equation*}
\begin{split}
|\mathbf{b}|^2&=2\int_0^y\mathbf{b}\cdot\mathbf{b}_{\xi}(\xi,t)d\xi\\
&\leq
C\int_0^1\frac{\nu|\mathbf{b}_y|^2}{v\theta}dy+C\int_0^1v\theta|\mathbf{b}|^2dy,
\end{split}
\end{equation*}
or
\begin{equation*}
\begin{split}
\int_0^t||\mathbf{b}||^2_{L^{\infty}([0,1])}ds &\leq
C\int_0^t\int_0^1\frac{\nu|\mathbf{b}_y|^2}{v\theta}dyds+C\int_0^t\max_{[0,1]}\theta
\left(\int_0^1v|\mathbf{b}|^2dy\right)ds.\\
&\leq C(T).
\end{split}
\end{equation*}
This finishes the proof.
\end{proof}
With these preliminary processes, the lower bound of density can be
inferred that
\begin{lemma}One has
\begin{equation}\label{lowerbounddensity}
\rho(y,t)\geq  C(T).
\end{equation}
\end{lemma}
\begin{proof}
Recalling the equality \eqref{lemma2V}, we find that
\begin{equation*}
\begin{split}
\lambda \ln v&=\lambda \ln v_0(y)+\int_0^t(p+\frac12|\mathbf{b}|^2)d\tau
+\int_0^y(u-u_0)dx\\
&\leq
C(T)+C\int_0^t\max_{[0,1]}\theta^4ds+C\int_0^t||\mathbf{b}||_{L^{\infty}([0,1])}^2ds+C\int_0^1u^2dy\\
&\leq C(T),\label{}
\end{split}
\end{equation*}
where we have used Lemmas \ref{Lemma1} and \ref{Lemma6}.
\end{proof}

\subsection{Upper and lower bounds of temperature}
At first, we are in a position to deduce   a priori estimates of the derivatives of
$(u,\mathbf{w},\mathbf{b})$.
\begin{lemma}One has
\begin{equation}
\int_0^t\int_0^1\left(u_y^2+|\mathbf{w}_y|^2+|\mathbf{b}_y|^2\right)dyds\leq
C(T).\label{lemma881}
\end{equation}
\end{lemma}
\begin{proof}
 The equality \eqref{sub55}
can be rewritten as
\begin{equation}
e_t+pu_y=\frac{\lambda u_y^2}{v}+\frac{\mu
|\mathbf{w}_y|^2}{v}+\frac{\nu
|\mathbf{b}_y|^2}{v}+\left(\frac{\kappa}{v}\theta_y\right)_y.
\label{lemma81}
\end{equation}
Then integrating \eqref{lemma81} over $[0,1]\times [0,t]$, we
get from \eqref{coro1} and \eqref{lemma61}
\begin{equation*}
\begin{split}
&\int_0^t\int_0^1\left(\frac{\lambda u_y^2}{v}+\frac{\mu
|\mathbf{w}_y|^2}{v}+\frac{\nu |\mathbf{b}_y|^2}{v}\right)dyds\\
&=\int_0^1(e-e_0)dy+\int_0^t\int_0^1pu_ydyds\\
&\leq C+\frac12\int_0^t\int_0^1\frac{\lambda
u_y^2}{v}dyds+C\int_0^t\int_0^1p^2dyds\\
&\leq C+\frac12\int_0^t\int_0^1\frac{\lambda u_y^2}{v}dyds+
C\int_0^t\max_{[0,1]}\theta^4\left(\int_0^1\theta^4dy\right)ds\\
&\leq C(T)+\frac12\int_0^t\int_0^1\frac{\lambda u_y^2}{v}dyds,
\end{split}
\end{equation*}
which ends the proof according to \eqref{lemma61} and \eqref{lowerbounddensity}.
\end{proof}
\begin{lemma}\label{Lemma9}The following inequalities hold for $\mathbf{b}$
\begin{equation}
\int_0^t\int_0^1|\mathbf{b}\cdot\mathbf{b}_y|^2dyds\leq
C(T),\label{lemma92}
\end{equation}
 and
\begin{equation}
 \int_0^t\int_0^1|\mathbf{b}|^8dyds\leq C(T).\label{lemma93}
\end{equation}
when $t\in [0,T]$.
\end{lemma}
\begin{proof}
By \eqref{magnetic}, one has
\begin{equation*}
\begin{split}
&\int_0^t\int_0^1|\mathbf{b}|^8dyds\\
&\leq \int_0^t\max_{[0,1]}|\mathbf{b}|^6\left(\int_0^1|\mathbf{b}|^2dy\right)ds\\
&\leq
C\int_0^t\int_0^1|\mathbf{b}|^6dyds+C\int_0^t\int_0^1|\mathbf{b}|^4|\mathbf{b}\cdot\mathbf{b}_y|dyds\\
&\leq
\frac12\int_0^t\int_0^1|\mathbf{b}|^8dyds+C\int_0^t\int_0^1\frac{|\mathbf{b}\cdot\mathbf{b}_y|^2}{v}dyds+C(T),
\end{split}
\end{equation*}
which implies
\begin{equation}
\int_0^t\int_0^1|\mathbf{b}|^8dyds\leq
C\int_0^t\int_0^1\frac{|\mathbf{b}|^2|\mathbf{b}_y|^2}{v}dyds+C(T).\label{lemma91}
\end{equation}
Multiplying $4|\mathbf{b}|^2\mathbf{b}$ on both sides of
\eqref{sub44}, one has
\begin{equation*}
\begin{split}
(v|\mathbf{b}|^4)_t=\left(\frac{\nu
\mathbf{b}_y}{v}+\mathbf{w}\right)_y\cdot
4|\mathbf{b}|^2\mathbf{b}-3v_t|\mathbf{b}|^4,
\end{split}
\end{equation*}
which follows  from \eqref{lemma62} and \eqref{lemma881} that
\begin{equation*}
\begin{split}
&\int_0^1v|\mathbf{b}|^4dy+12\nu\int_0^t\int_0^1\frac{|\mathbf{b}|^2|\mathbf{b}_y|^2}{v}dyds\\
&=\int_0^1v|\mathbf{b}|^4(y,0)dy-12\int_0^t\int_0^1|\mathbf{b}|^2\mathbf{w}\cdot\mathbf{b}_ydyds
-3\int_0^t\int_0^1u_y|\mathbf{b}|^4dyds\\
&\leq \varepsilon
\int_0^t\int_0^1\frac{|\mathbf{b}|^2|\mathbf{b}_y|^2}{v}dyds
+C(\varepsilon)\int_0^t||\mathbf{b}||_{L^{\infty}([0,1])}^2\left(\int_0^1|\mathbf{w}|^2dy\right)ds\\
&+\varepsilon\int_0^t\int_0^1|\mathbf{b}|^8dyds+C(\varepsilon)\int_0^t\int_0^1u_y^2dyds+C(T)\\
&\leq \varepsilon
\int_0^t\int_0^1\frac{|\mathbf{b}|^2|\mathbf{b}_y|^2}{v}dyds+\varepsilon\int_0^t\int_0^1|\mathbf{b}|^8dyds+C(T).
\end{split}
\end{equation*}
Substituting  \eqref{lemma91} into the above inequality, we deduce
\eqref{lemma92} and \eqref{lemma93} by taking sufficiently small $\varepsilon$.
\end{proof}

In the following, we first verify that the  temperature is dominated by velocity and magnetic fields.
\begin{lemma}\label{lemma09} One has
\begin{equation*}
\begin{split}
&\int_0^1\theta^8dy+\int_0^t\int_0^1\kappa \theta^3\theta_y^2dyds\\
&\leq C(T)+
C(T) \left(\int_0^t\int_0^1\mathbf{w}_{yy}^2dyds\right)^{\frac12}+C(T)
\left(\int_0^t\int_0^1\mathbf{b}_{yy}^2dyds\right)^{\frac12}.
\end{split}
\end{equation*}
\end{lemma}
\begin{proof}
Multiplying $\theta^4$ on both sides of \eqref{sub55}, and then
integrating it over $[0,1]\times [0,t]$, it leads to
\begin{equation*}
\begin{split}
&\int_0^1\theta^8dy+\int_0^t\int_0^1\kappa
\theta^3\theta_y^2dyds\\
&\leq
C\int_0^t\int_0^1(\theta^8|u_y|+\theta^4u_y^2)dyds+C\int_0^t\int_0^1\left(\frac{\mu
\mathbf{w}_y^2}{v}+\frac{\nu \mathbf{b}_y^2}{v}\right)\cdot
\theta^4dyds\\
&\leq
C+C\int_0^t\int_0^1(\theta^{12}+|u_y|^3)dyds+C\int_0^t\left(||\mathbf{w}_y||_{L^{\infty}}^2+||\mathbf{b}_y||_{L^{\infty}}^2\right)\int_0^1\theta^4dyds\\
&\leq
C+C\int_0^t\int_0^1(\theta^{12}+|u_y|^3)dyds\\
&+C\left(\int_0^t\int_0^1\mathbf{w}_{yy}^2dyds\right)^{\frac12}+C
\left(\int_0^t\int_0^1\mathbf{b}_{yy}^2dyds\right)^{\frac12}.
\end{split}
\end{equation*}
On the other hand, we get from \eqref{sub22}
\begin{equation}
\left\{
\begin{array}{llllllllll}
h_t=\frac{\lambda}{v}h_{yy}-(p+\frac12|\mathbf{b}|^2),\\[2mm]
h |_{t=0}=h_0(y),\\[2mm]
h|_{y=0,1}=0,\label{LP}
\end{array}
\right.
\end{equation}
where $h=\int_0^yu d\xi$. The standard $L^p$ estimates of solutions
to linear parabolic problem yield
\begin{equation*}
\begin{split}
\int_0^t\int_0^1|u_y|^3dyds
&=\int_0^t\int_0^1|h_{yy}|^3dyds\\
&\leq C \left(1+\int_0^t\int_0^1 (p^3+|\mathbf{b}|^6)dyds\right)\\
&\leq C(T) \left(1+\int_0^t\int_0^1 \theta^{12} dyds\right)
\end{split}
\end{equation*}
and then
\begin{equation*}
\begin{split}
&\int_0^1\theta^8dy+\int_0^t\int_0^1\kappa\theta^3\theta_y^2dyds\\
&\leq
C+C\left(\int_0^t\int_0^1\mathbf{w}_{yy}^2dyds\right)^{\frac12}+C
\left(\int_0^t\int_0^1\mathbf{b}_{yy}^2dyds\right)^{\frac12}\\
&+C\int_0^t\int_0^1\theta^{12}dyds+C\int_0^t\int_0^1|u_y|^3dyds\\
&\leq
C(T)\left(1+\int_0^t\int_0^1\theta^{12}dyds\right)+C\left(\int_0^t\int_0^1\mathbf{w}_{yy}^2dyds\right)^{\frac12}+C
\left(\int_0^t\int_0^1\mathbf{b}_{yy}^2dyds\right)^{\frac12}\\
&\leq
C(T)+C\int_0^t\max_{[0,1]}\theta^4\int_0^1\theta^8dyds\\
&+C\left(\int_0^t\int_0^1\mathbf{w}_{yy}^2dyds\right)^{\frac12}+C
\left(\int_0^t\int_0^1\mathbf{b}_{yy}^2dyds\right)^{\frac12}.
\end{split}
\end{equation*}
It finishes the proof by Gr\"{o}nwall inequality and \eqref{lemma61}.
\end{proof}
The following lemma also declare a relationship of density and velocity, as well as magnetic fields.
\begin{lemma}\label{lemma10}For any $\varepsilon>0$, it satisfies that
\begin{equation*}
\begin{split}
&\int_0^1v_y^2dy+\int_0^t\int_0^1\theta v_y^2dyds\\
&\leq C(T)
+\varepsilon\left(\int_0^t\int_0^1\mathbf{w}_{yy}^2dyds\right)^{\frac12}+\varepsilon
\left(\int_0^t\int_0^1\mathbf{b}_{yy}^2dyds\right)^{\frac12}.
\end{split}
\end{equation*}
\end{lemma}
\begin{proof}
We can rewrite the equation \eqref{sub22} as follows
\begin{equation}
\left(u-\frac{\lambda
v_y}{v}\right)_t=-\left(p+\frac12|\mathbf{b}|^2\right)_y.
\label{lemma101}
\end{equation}
Multiplying $\left(u-\frac{\lambda}{v}v_y\right)$ on both sides of
\eqref{lemma101} and integrating it over $(0,1)\times (0,t)$, we
find that
\begin{equation}
\begin{split}
&\frac12\int_0^1\left(u-\frac{\lambda}{v}v_y\right)^2dy+\int_0^t\int_0^1\frac{\lambda R \theta}{v^3}v_y^2dyds\\
=&\frac12\int_0^1\left(u-\frac{\lambda}{v}v_y\right)^2(y,0)dy+\int_0^t\int_0^1\frac{R\theta uv_y}{v^2}dyds\\
&-\int_0^t\int_0^1\left[\left(\frac{R}{v}
+\frac{4a}{3}\theta^3\right)\theta_y+\mathbf{b}\cdot\mathbf{b}_y
\right]\left(u-\frac{\lambda}{v}v_y\right)dyds.\label{lemma102}
\end{split}
\end{equation}
To complete the proof, it only evaluates all terms on the right hand
side of \eqref{lemma102}. Particularly,
\begin{equation}
\begin{split}
&\int_0^t\int_0^1\frac{R\theta uv_y}{v^2}dyds\\
&\leq \varepsilon \int_0^t\int_0^1\theta
v_y^2dyds+C_{\varepsilon}\int_0^t\max_{[0,1]}\theta\cdot\left(\int_0^1u^2dy\right)ds\\
&\leq \varepsilon \int_0^t\int_0^1\theta
v_y^2dyds+C_{\varepsilon}\int_0^t\max_{[0,1]}\theta(y,s)ds\\
&\leq \varepsilon \int_0^t\int_0^1\theta v_y^2dyds+C(T),
\label{lemma103}
\end{split}
\end{equation}
for any $\varepsilon >0$ and further
\begin{equation}
\begin{split}
&\left|\int_0^t\int_0^1\left[\left(\frac{R}{v}+\frac{4a}{3}\theta^3\right)\theta_y\right]
\left(u-\frac{\lambda}{v}v_y\right)dyds\right|\\[2mm]
&\leq  C
\int_0^t\int_0^1\frac{\kappa\theta_y^2}{\theta^2}dyds+\varepsilon\int_0^t\int_0^1\kappa\theta^3\theta_y^2dyds\\
&+\int_0^t\int_0^1\frac{\theta^2+\theta^3}{\kappa}\left(u-\frac{\lambda}{v}v_y\right)^2dyds\\[2mm]
&\leq  C+\varepsilon\int_0^t\int_0^1\kappa\theta^3\theta_y^2dyds
+\int_0^t\max_{[0,1]}\frac{\theta^2+\theta^3}{\kappa}\cdot\int_0^1\left(u-\frac{\lambda}{v}v_y\right)^2dyds.
\end{split}
\end{equation}
In addition, by \eqref{lemma92}, we also obtain
\begin{equation}
\begin{split}
&\left|\int_0^t\int_0^1\mathbf{b}\cdot\mathbf{b}_y\cdot\left(u-\frac{\lambda}{v}v_y\right)dyds\right|\\
&\leq
C\int_0^t\int_0^1|\mathbf{b}\cdot\mathbf{b}_y|^2dyds+C\int_0^t\int_0^1\left(u-\frac{\lambda}{v}v_y\right)^2dyds\\
&\leq
C(T)+C\int_0^t\int_0^1\left(u-\frac{\lambda}{v}v_y\right)^2dyds.\label{lemma104}
\end{split}
\end{equation}
Finally, substituting \eqref{lemma103}--\eqref{lemma104} into
\eqref{lemma102}, it leads to
\begin{equation*}
\begin{split}
&\int_0^1\left(u-\frac{\lambda}{v}v_y\right)^2dy+\int_0^t\int_0^1\theta v_y^2dyds\\
&\leq
C(T)+\varepsilon\int_0^t\int_0^1\kappa\theta^3\theta_y^2dyds\\
&+C\int_0^t\left(\max_{[0,1]}\frac{\theta^2+\theta^3}{\kappa}+1\right)
\cdot\int_0^1\left(u-\frac{\lambda}{v}v_y\right)^2dyds,
\end{split}
\end{equation*}
which ends the proof by Gr\"{o}nwall inequality and Lemma
\ref{lemma09} .
\end{proof}
Combining Lemma \ref{lemma09} and Lemma \ref{lemma10}, we can give a priori estimates of velocity and magnetic fields,
which play an important role to include the constant heat conductivity.
\begin{lemma}It holds that
\begin{equation}\label{freelemma131}
\begin{split}
&||\mathbf{b}||_{L^{\infty}((0,1)\times(0,T))}+\int_0^1(|\mathbf{b}_y|^2+|\mathbf{w}_y|^2)dy\\
& +\int_0^t\int_0^1(|\mathbf{b}_y|^4+|\mathbf{w}_y|^4+|\mathbf{b}_{yy}|^2+|\mathbf{w}_{yy}|^2)dyds
\leq
C(T).
\end{split}
\end{equation}
\end{lemma}
\begin{proof}
Multiplying \eqref{sub33} by $\mathbf{w}_{yy}$ and integrating it
over $[0,1]\times [0,t]$, one has
\begin{equation*}
\begin{split}
&\frac12\int_0^1|\mathbf{w}_y|^2dy=\frac12\int_0^1|\mathbf{w}_y(y,0)|^2dy
-\int_0^t\int_0^1\left(\mathbf{b}+\frac{\mu
\textbf{w}_y}{v}\right)_y\cdot\mathbf{w}_{yy}dyds\\
&\leq C-C(T)\int_0^t\int_0^1|\mathbf{w}_{yy}|^2dyds+
C\int_0^t\int_0^1(|\mathbf{b}_y|+|v_y||\mathbf{w}_y|)|\mathbf{w}_{yy}|dyds\\
&\leq C-\frac{3C(T)}{4}\int_0^t\int_0^1|\mathbf{w}_{yy}|^2dyds
+C(T)\int_0^t\int_0^1|\mathbf{b}_y|^2dyds\\
&+C(T)\int_0^t\max_{[0,1]}|\mathbf{w}_y|^2\left(\int_0^1v_y^2dy\right)ds\\
&\leq
C-\frac{3C(T)}{4}\int_0^t\int_0^1|\mathbf{w}_{yy}|^2dyds\\
&+\left(\varepsilon\left(\int_0^t\int_0^1\mathbf{w}_{yy}^2dyds\right)^{\frac12}+\varepsilon
\left(\int_0^t\int_0^1\mathbf{b}_{yy}^2dyds\right)^{\frac12}\right)\times\left(\int_0^t\int_0^1\mathbf{w}_{yy}^2dyds\right)^{\frac12}\\
&\leq
C-\frac{C(T)}{2}\int_0^t\int_0^1|\mathbf{w}_{yy}|^2dyds+\varepsilon
\int_0^t\int_0^1\mathbf{b}_{yy}^2dyds.
\end{split}
\end{equation*}
 Thus,
\begin{equation}\label{wb}
\int_0^1|\mathbf{w}_y|^2dy+\int_0^t\int_0^1|\mathbf{w}_{yy}|^2dyds\leq
C+\varepsilon \int_0^t\int_0^1\mathbf{b}_{yy}^2dyds,
\end{equation}

On the other hand, we can also rewrite \eqref{sub44} as follows
\begin{equation*}
\mathbf{b}_t=-\frac{u_y}{v}\mathbf{b}+\frac{1}{v}\left(\mathbf{w}+\frac{\nu
\mathbf{b}_y}{v}\right)_y.
\end{equation*}
Following the same procedure as above, we get from \eqref{magnetic}
\begin{equation*}
\begin{split}
&\frac12\int_0^1|\mathbf{b}_y|^2dy\\
&\leq C-C\int_0^t\int_0^1|\mathbf{b}_{yy}|^2dyds\\
&+C\int_0^t\int_0^1(|u_y||\mathbf{b}|+|\mathbf{w}_y|+|v_y||\mathbf{b}_y|)|\mathbf{b}_{yy}|dyds\\
&\leq C-\frac{3C}{4}\int_0^t\int_0^1|\mathbf{b}_{yy}|^2dyds
+C\sup_{(y,t)\in(0,1)\times(0,t)}|\mathbf{b}|^2\int_0^t\int_0^1u_y^2dyds\\
&+C\int_0^t\max_{[0,1]}|\mathbf{b}_y|^2\bigg(\int_0^1v_y^2dy\bigg)ds\\
&\leq
C-\frac{3C}{4}\int_0^t\int_0^1|\mathbf{b}_{yy}|^2dyds+C\int_0^1|\mathbf{b}|^2dy+\frac14\int_0^1|\mathbf{b}_y|^2dy\\
&+\left(\varepsilon\left(\int_0^t\int_0^1\mathbf{w}_{yy}^2dyds\right)^{\frac12}+\varepsilon
\left(\int_0^t\int_0^1\mathbf{b}_{yy}^2dyds\right)^{\frac12}\right)\times \left(\int_0^t\int_0^1\mathbf{b}_{yy}^2dyds\right)^{\frac12}\\
&\leq
C(T)-\frac{C}{2}\int_0^t\int_0^1|\mathbf{b}_{yy}|^2dyds+\frac14\int_0^1|\mathbf{b}_y|^2dy+\varepsilon\int_0^t\int_0^1\mathbf{w}_{yy}^2dyds,
\end{split}
\end{equation*}
which implies
\begin{equation*}
\int_0^1|\mathbf{b}_y|^2dy+\int_0^t\int_0^1|\mathbf{b}_{yy}|^2dyds\leq
C(T)+\varepsilon\int_0^t\int_0^1\mathbf{w}_{yy}^2dyds.
\end{equation*}
This together with \eqref{wb} leads to
\begin{equation*}
\int_0^1(|\mathbf{w}_y|^2+|\mathbf{b}_y|^2)dy+\int_0^t\int_0^1(|\mathbf{w}_{yy}|^2+|\mathbf{b}_{yy}|^2)dyds\leq
C(T).
\end{equation*}
Similarly, we get
\begin{equation*}
\begin{split}
\int_0^t\int_0^1|\mathbf{w}_y|^4dyds&\leq
\int_0^t\max_{[0,1]}|\mathbf{w}_y|^2\left(\int_0^1|\mathbf{w}_y|^2dy\right)ds\\
&\leq C(T)\int_0^t\int_0^1(|\mathbf{w}_y|^2+|\mathbf{w}_{yy}|^2)dyds\\
&\leq C(T).
\end{split}
\end{equation*}
and
\begin{equation*}
\begin{split}
\int_0^t\int_0^1|\mathbf{b}_y|^4dyds&\leq
\int_0^t\max_{[0,1]}|\mathbf{b}_y|^2\left(\int_0^1|\mathbf{b}_y|^2dy\right)ds\\
&\leq C(T)\int_0^t\int_0^1(|\mathbf{b}_y|^2+|\mathbf{b}_{yy}|^2)dyds\\
&\leq C(T).
\end{split}
\end{equation*}
This completes the proof of the lemma.
\end{proof}
With \eqref{freelemma131} in hand, we thereby can check Lemma \ref{lemma09} and Lemma \ref{lemma10} again and deduce that
\begin{corollary}We have
\begin{equation}
\int_0^1v_y^2dy+\int_0^t\int_0^1\theta v_y^2dyds\leq C(T).\label{density}
\end{equation}
and
\begin{equation}
\int_0^t\max_{[0,1]}\theta^{q+13}ds+\int_0^1\theta^8dy+\int_0^t\int_0^1\kappa
\theta^3\theta_y^2dyds\leq C(T).\label{temperature}
\end{equation}
\end{corollary}
Similarly, using $L^p$ estimates of solutions to linear parabolic
problem again, we get from \eqref{lemma93} and  \eqref{LP}
\begin{equation}\label{velocity}
\begin{split}
\int_0^t\int_0^1|u_y|^4dyds
&=\int_0^t\int_0^1|w_{yy}|^4dyds\\
&\leq C \left(1+\int_0^t\int_0^1 (p^3+|\mathbf{b}|^8)dyds\right)\\
&\leq C \left(1+\int_0^t\int_0^1 \theta^{16} dyds\right)\\
&\leq C \left(1+\int_0^t\max_{[0,1]}\theta^8\int_0^1 \theta^{8}dyds\right)\\
&\leq C(T).
\end{split}
\end{equation}

In order  to get the upper bound of temperature,  we also need to establish higher order  a priori estimates of
$(v,u,\theta,\mathbf{w},\mathbf{b})$. Thus, we introduce some auxiliary
new variables motivated by \cite{kawohl}.
\begin{align}
&X:=\int_0^t\int_0^1(1+\theta^{q})\theta_t^2dyds,\label{X}\\
&Y:=\max_{0\leq t\leq
T}\int_0^1(1+\theta^{2q})\theta_y^2dy,\label{Y}\\
&Z:=\max_{0\leq t\leq T}\int_0^1u_{yy}^2dy.\label{Z}
\end{align}
By interpolation and embedding theorem, It follows from \eqref{Z} that
\begin{align}\label{boundvelocity}
&|u_y|^{(0)}\leq C(1+Z^{\frac38}),
\end{align}
where $|\cdot|^{(0)}=\sup|\cdot|$.

\begin{lemma}
Under the assumptions of Theorem \ref{thm}, we have
\begin{equation}
\max_{[0,1]\times[0,t]}\theta(y,s)\leq
C+CY^{\frac{1}{2(q+5)}},\label{temperature1}
\end{equation}
for $(y,t)\in (0,1)\times (0,T)$.
\end{lemma}
\begin{proof}
By the embedding theorem, we deduce that
\begin{equation*}
\max_{[0,1]}\theta^{q+5}(y,t)  \leq
C\int_0^1\theta^{q+5}dy+C\int_0^1(1+\theta)^{q+4}|\theta_y|dy,
\end{equation*}
which implies by \eqref{temperature} and H\"{o}lder inequality,
\begin{equation*}
\begin{split}
&\max_{[0,1]}\theta^{q+5}(y,t)\\
&\leq C\max_{[0,1]}\theta^{q+1}\int_0^1\theta^4
dy+C\int_0^1(1+\theta)^{q}|\theta_y|(1+\theta)^4
dy\\
&\leq \frac12 \max_{[0,1]}\theta^{q+5}
+C\left(\int_0^1(1+\theta)^{2q}\theta_y^2dy\right)^{\frac12}
\left(\int_0^1(1+\theta)^8dy\right)^{\frac12}+C\\
&\leq \frac12\max_{[0,1]} \theta^{q+5}+C(T)
\left(\int_0^1(1+\theta)^{2q}\theta_y^2dy\right)^{\frac12}+C(T)\\
&\leq\frac12\max_{[0,1]}\theta^{q+5} +CY^{\frac12}+C(T).
\end{split}
\end{equation*}
It finishes the proof.
\end{proof}


\begin{lemma}\label{XY}Furthermore, we have
\begin{equation}
X+Y\leq C(T)\left(1+Z^{\frac{q+5}{q+10}}\right). \label{Lemma14}
\end{equation}
\end{lemma}
\begin{proof} We introduce the function as in \cite{kawohl,umeharaTani}
\begin{equation*}
K(v,\theta):=\int_0^{\theta}\frac{\kappa(v,\xi)}{v}d\xi.
\end{equation*}
The simple calculation leads to
\begin{align}
&K_t=\frac{\kappa}{v}\theta_t+K_vu_y, \label{K1}\\
&K_{yt}=\left(\frac{\kappa}{v}\theta_y\right)_t+K_{vv}v_{y}u_{y}+K_{v}u_{yy}+\left(\frac{\kappa}{v}\right)_vv_y
\theta_t,\label{K2}\\
&|K_v|, |K_{vv}|\leq C(1+\theta^{q+1}).\label{K3}
\end{align}
Multiplying \eqref{sub55} by $K_t$ and integrating it over
$(0,1)\times (0,t)$, we get
\begin{equation*}
\begin{split}
&\int_0^t\int_0^1\left(e_{\theta}\theta_t+\theta
p_{\theta}u_y-\frac{\lambda u_y^2}{v}-\frac{\mu
|\mathbf{w}_y|^2}{v}-\frac{\nu
|\mathbf{b}_y|^2}{v}\right)K_tdyds\\
&+\int_0^t\int_0^1\frac{\kappa}{v}\theta_y K_{yt}dyds=0,
\end{split}
\end{equation*}
or equivalently
\begin{equation}
\begin{split}
&\int_0^t\int_0^1\frac{\kappa e_{\theta}\theta_t^2}{v}dyds
+\int_0^t\int_0^1\frac{\kappa}{v}\theta_y\left(\frac{\kappa}{v}\theta_y\right)_tdyds\\
&=I_1+I_2+I_3+I_4+I_5+I_6, \label{lemma141}
\end{split}
\end{equation}
where $I_{i},(i=1,\cdots,6)$ is defined as follows
\begin{equation*}
\begin{split}
&I_1=-\int_0^t\int_0^1e_{\theta}\theta_tK_v u_ydyds\\[2mm]
&I_2=-\int_0^t\int_0^1\left(\theta p_{\theta}u_y-\frac{\lambda
u_y^2}{v}-\frac{\mu |\mathbf{w}_y|^2}{v}-\frac{\nu
|\mathbf{b}_y|^2}{v}\right)\frac{\kappa}{v}\theta_tdyds\\[2mm]
&I_3=- \int_0^t\int_0^1\left(\theta p_{\theta}u_y-\frac{\lambda
u_y^2}{v}-\frac{\mu |\mathbf{w}_y|^2}{v}-\frac{\nu
|\mathbf{b}_y|^2}{v}\right)K_vu_ydyds\\[2mm]
&I_4=-\int_0^t\int_0^1\frac{\kappa}{v}\theta_yK_{vv}v_yu_ydyds\\[2mm]
&I_5=-\int_0^t\int_0^1\frac{\kappa}{v}\theta_y K_v u_{yy}dyds\\[2mm]
&I_6=-\int_0^t\int_0^1\frac{\kappa}{v}\theta_y\left(\frac{\kappa}{v}\right)_v
v_{y}\theta_tdyds.
\end{split}
\end{equation*}

In the sequel, we will evaluate all terms
\eqref{lemma141} according to \eqref{K1}--\eqref{K3}.

Firstly, one has by \eqref{X} and \eqref{Y}
\begin{equation}
\int_0^t\int_0^1\frac{\kappa e_{\theta}\theta_t^2}{v}dyds \geq C
\int_0^t\int_0^1(1+\theta^3)(1+\theta^q)\theta_t^2dyds\geq CX,
\end{equation}
and
\begin{equation}
\begin{split}
&\int_0^t\int_0^1\frac{\kappa}{v}\theta_y\left(\frac{\kappa}{v}\theta_y\right)_{t}dyds\\
&=\frac12\int_0^1\left(\frac{\kappa}{v}\theta_y\right)^2dy
-\frac12\int_0^1\left(\frac{\kappa}{v}\theta_y\right)^2(y,0)dy\\
&\geq C\int_0^1(1+\theta^{q})^2\theta_y^2dy-C\geq CY-C.
\end{split}
\end{equation}
Secondly, we have by Cauchy-Schawtz inequality and \eqref{temperature1}
\begin{equation}\label{lemma105}
\begin{split}
&\left|I_1\right|=\left|\int_0^t\int_0^1e_{\theta}\theta_tK_v u_ydyds\right| \\
&\leq C\int_0^t\int_0^1(1+\theta)^{q+4}|\theta_t||u_y|dyds\\
&\leq \varepsilon X+C_{\varepsilon}\int_0^t\int_0^1(1+\theta)^{q+8}u_y^2dyds\\
&\leq \varepsilon
X+C_{\varepsilon}Y^{\frac{q+8}{2(q+5)}}\int_0^t\int_0^1u_y^2dyds\\
&\leq \varepsilon (X+Y)+C,
\end{split}
\end{equation}
for any fixed sufficiently small $\varepsilon>0$, and similarly, by recalling \eqref{freelemma131} and \eqref{velocity}, one has
\begin{equation}
\begin{split}
&\left|I_2\right|=\left|\int_0^t\int_0^1\left(\theta
p_{\theta}u_y-\frac{\lambda u_y^2}{v}-\frac{\mu
|\mathbf{w}_y|^2}{v}-\frac{\nu
|\mathbf{b}_y|^2}{v}\right)\frac{\kappa}{v}\theta_tdyds\right|\\
&\leq C\int_0^t\int_0^1\left[(1+\theta)^{q+4}|u_y\theta_t|
+(1+\theta)^q|\theta_t|(u_y^2+|\mathbf{w}_y|^2+|\mathbf{b}_y|^2)\right]dyds\\
&\leq \varepsilon
X+C_{\varepsilon}|(1+\theta)^{q+8}|^{(0)}\int_0^t\int_0^1u_y^2dyds\\
&+C_{\varepsilon}|(1+\theta)^{q}|^{(0)}\int_0^t\int_0^1\left(|u_y|^4+|\mathbf{w}_y|^4+|\mathbf{b}_y|^4\right)dyds\\
&\leq \varepsilon (X+Y)+C.
\end{split}
\end{equation}
 and
\begin{equation}
\begin{split}
&\left|I_3\right|=\left|\int_0^t\int_0^1\left(\theta
p_{\theta}u_y-\frac{\lambda u_y^2}{v}-\frac{\mu
|\mathbf{w}_y|^2}{v}-\frac{\nu
|\mathbf{b}_y|^2}{v}\right)K_vu_ydyds\right|\\
&\leq
C\int_0^t\int_0^1(1+\theta)^{q+5}u_y^2dyds+\int_0^t\int_0^1(1+\theta)^{q+1}
|u_y|^3dyds\\
&\qquad+C\int_0^t\int_0^1(1+\theta)^{q+1}|u_y|(|\mathbf{w}_y|^2+|\mathbf{b}_y|^2)dyds\\
&\leq C
|(1+\theta)^{q+5}|^{(0)}\int_0^t\int_0^1u_y^2dyds\\
&\qquad+|(1+\theta)^{q+1}|^{(0)}\int_0^t\int_0^1|u_y|^3dyds\\
&\qquad+C\int_0^t\int_0^1(1+\theta)^{q+1}|u_y|(|\mathbf{w}_y|^2+|\mathbf{b}_y|^2)dyds\\
&\leq CY^{\frac{q+5}{2(q+5)}}+
CY^{\frac{q+1}{2(q+5)}}\int_0^t\int_0^1(u_y^2+|\mathbf{w}_y|^4+|\mathbf{b}_y|^4)dyds+C\\
&\leq \varepsilon Y+C.
\end{split}
\end{equation}
Thirdly, it also yields in view of \eqref{lemma61} and \eqref{boundvelocity}
\begin{equation}
\begin{split}
&\left|I_4\right|=\left|\int_0^t\int_0^1\frac{\kappa}{v}\theta_yK_{vv}v_yu_ydyds\right|\\
&\leq
C\left(\int_0^t\int_0^1\frac{\kappa\theta_y^2}{\theta^2}dyds\right)^{\frac12}
\times\left(\int_0^t\int_0^1(1+\theta)^{3q+4}u_y^2v_y^2dyds\right)^{\frac12}\\
&\leq CZ^{\frac38}Y^{\frac{q}{2(q+5)}}\\
 &\leq \varepsilon Y+CZ^{\frac{3(q+5)}{4(q+10)}},
\end{split}
\end{equation}
and
\begin{equation}\label{lemma142}
\begin{split}
&\left|I_5\right|=\left|\int_0^t\int_0^1\frac{\kappa}{v}\theta_y K_v u_{yy}dyds\right|\\
&\leq C
\left(\int_0^t\int_0^1\frac{\kappa\theta_y^2}{\theta^2}dyds\right)^{\frac12}\cdot
\left(\int_0^t\int_0^1(1+\theta)^{(3q+4)}u_{yy}^2dyds\right)^{\frac12}\\
&\leq CZ^{\frac12}Y^{\frac{q}{2(q+5)}}\\
 &\leq \varepsilon
Y+CZ^{\frac{q+5}{q+10}}.
\end{split}
\end{equation}
Finally, one has
\begin{equation}\label{lemma106}
\begin{split}
&\left|I_6\right|=\left|\int_0^t\int_0^1\frac{\kappa}{v}\theta_y\left(\frac{\kappa}{v}\right)_v
v_{y}\theta_tdyds\right|\\
&\leq
C\int_0^t\int_0^1\left|\frac{\kappa}{v}\theta_y\right|(1+\theta)^q
|v_{y}\theta_t|dyds \\
&\leq \varepsilon
X+C_{\varepsilon}|(1+\theta)^{q}|^{(0)}\int_0^t\max_{[0,1]}\left(\frac{\kappa\theta_y}{v}\right)^2
\left(\int_0^1v_y^2dy\right)ds\\
&\leq \varepsilon X+C_{\varepsilon}Y^{\frac{q}{2(q+5)}}
\int_0^t\max_{[0,1]}\left(\frac{\kappa\theta_y}{v}\right)^2\left(\int_0^1v_y^2dy\right)ds\\
&\leq \varepsilon
X+C_{\varepsilon}Y^{\frac{q}{2(q+5)}}\int_0^t\max_{[0,1]}\left(\frac{\kappa\theta_y}{v}\right)^2ds.
\end{split}
\end{equation}

By the embedding theorem again, one has
\begin{equation*}
\begin{split}
&\int_0^t\max_{[0,1]}\left(\frac{\kappa\theta_y}{v}\right)^2ds\\
&\leq C\int_0^t\int_0^1\left(\frac{\kappa\theta_y}{v}\right)^2dyds
+C\int_0^t\int_0^1\left|\frac{\kappa\theta_y}{v}\cdot\left(\frac{\kappa\theta_y}{v}\right)_y\right|dyds\\
&\leq C|\kappa \theta^{2}|^{(0)}\int_0^t\int_0^1\frac{\kappa\theta_y^2}{\theta^2}dyds\\
&~~+C\left(\int_0^t\int_0^1\frac{\kappa\theta_y^2}{\theta^2}dyds\right)^{\frac12}\left(\int_0^t\int_0^1\kappa\theta^{2}
\left(\frac{\kappa\theta_y}{v}\right)_y^2dyds\right)^{\frac12}\\
 &\leq C
\bigg\{|(1+\theta)^{q+2}|^{(0)}\\
&~~+\left(\int_0^t\int_0^1(1+\theta)^{q+2}\left(e_{\theta}^2\theta_t^2
 +\theta^2
p_{\theta}^2u_y^2+u_y^4+|\mathbf{w}_y|^4+|\mathbf{b}_y|^4\right)dyds\right)^{\frac12}\bigg\}.
\end{split}
\end{equation*}
In particular, we also have
\begin{equation*}
\begin{split}
&\int_0^t\int_0^1(1+\theta)^{q+2}e_{\theta}^2\theta_{t}^2dyds\\
&\leq
C|(1+\theta)^{8}|^{(0)}\int_0^t\int_0^1(1+\theta)^{q}\theta_{t}^2dyds\\[2mm]
&\leq C \left(X+Y^{\frac{8}{2(q+5)}}X\right),
\end{split}
\end{equation*}
 and
\begin{equation*}
\begin{split}
&\int_0^t\int_0^1(1+\theta)^{q+2}\theta^2 p_{\theta}^2u_y^2dyds\\
&\leq C\int_0^t\int_0^1(1+\theta)^{q+10}u_y^2dyds\\
&\leq
C|(1+\theta)^{q+10}|^{(0)}\int_0^t\int_0^1u_y^2dyds\\
&\leq C\left(1+Y^{\frac{q+10}{2(q+5)}}\right),
\end{split}
\end{equation*}
and
\begin{equation*}
\begin{split}
&\int_0^t\int_0^1(1+\theta)^{q+2}\left(u_y^4+|\mathbf{w}_y|^4+|\mathbf{b}_y|^4\right)dyds\\
&\leq C|(1+\theta)^{q+2}|^{(0)}\int_0^t\int_0^1\left(u_y^4+|\mathbf{w}_y|^4+|\mathbf{b}_y|^4\right)dyds\\
&\leq C\left(1+Y^{\frac{q+2}{2(q+5)}}\right),
\end{split}
\end{equation*}
which implies
\begin{equation*}
\begin{split}
&Y^{\frac{q}{2(q+5)}}\int_0^t\max_{[0,1]}\left(\frac{\kappa
\theta_y}{v}\right)^2ds
\\&\leq
C\left(1+Y^{\frac{q+1}{q+5}}+X^{\frac12}Y^{\frac{q+4}{2(q+5)}}
+Y^{\frac{3q+10}{4(q+5)}}\right)\\[2mm]
&\leq \varepsilon (X+Y),
\end{split}
\end{equation*}
which together with \eqref{lemma106} shows that
\begin{equation}
 |I_6|\leq \varepsilon (X+Y).
 \end{equation}
 This together with \eqref{lemma141}--\eqref{lemma142} completes the proof.
\end{proof}
With the above a priori estimates, we can succeed in obtaining the upper bound of temperature, which is another essential contribution of the paper.
\begin{lemma}
The following  a priori estimates hold,
\begin{equation}
|\theta|^{(0)}+|u_y|^{(0)}+|u|^{(0)}\leq C(T).\label{lemma17theta}
\end{equation}
and
\begin{equation}\label{lemma17}
\int_0^1(\theta_y^2+u_{yy}^2+
u_t^2)dy+\int_0^t\int_0^1(\theta_t^2+|\mathbf{b}_t|^2+u_{yt}^2)dyds\leq
C(T),
\end{equation}
\end{lemma}
\begin{proof}
Differentiate \eqref{sub22} with respect to $t$, multiply it by
$u_t$, and then integrate to obtain
\begin{equation*}
\frac{d}{dt}\int_0^1\frac{u_t^2}{2}dy+\lambda
\int_0^1\frac{u_{yt}^2}{v}dy
=\int_0^1\left(\left(p+\frac12|\mathbf{b}|^2\right)_t+\frac{\lambda u_y^2}{v^2}\right)u_{yt}dy.
\end{equation*}
As a consequence, it  implies that
\begin{equation*}
\begin{split}
&\int_0^1u_t^2dy+\int_0^t\int_0^1u_{yt}^2dyds\\
&\leq
C+C\int_0^t\int_0^1u_y^2|u_{yt}|dyds+C\int_0^t\int_0^1|p_tu_{yt}|dyds
+C\int_0^t\int_0^1|\mathbf{b}\cdot\mathbf{b}_tu_{yt}|dyds\\
&\leq
\frac34\int_0^t\int_0^1u_{yt}^2dyds+C\int_0^t\int_0^1u_y^4dyds\\
&+C\int_0^t\int_0^1p_t^2dyds
+C\int_0^t\int_0^1|\mathbf{b}_t|^2dyds+C(T)\\
&\leq C(T)+\frac34\int_0^t\int_0^1u_{yt}^2dyds
+C\int_0^t\int_0^1p_t^2dyds
+C\int_0^t\int_0^1|\mathbf{b}_t|^2dyds
\end{split}
\end{equation*}
which leads to
\begin{equation}\label{lemma150}
\begin{split}
&\int_0^1u_t^2dy+\int_0^t\int_0^1u_{yt}^2dyds\\
&\leq C(T)+C\int_0^t\int_0^1p_t^2dyds
+C\int_0^t\int_0^1|\mathbf{b}_t|^2dyds.
\end{split}
\end{equation}
We notice that
\begin{equation}
\begin{split}
&\int_0^t\int_0^1p_t^2dyds=\int_0^t\int_0^1\left(p_vv_t+p_{\theta}\theta_t\right)^2dyds\\
&=\int_0^t\int_0^1\left(-\frac{R\theta}{v^2}u_y+\left(R\rho+\frac{4a}{3}\theta^3\right)\theta_t\right)^2dyds\\
&\leq
C\int_0^t\int_0^1\theta^2u_y^2dyds+C\int_0^t\int_0^1(1+\theta^6)\theta_t^2dyds\\
&\leq
C(T)|\theta^2|^{(0)}\int_0^t\int_0^1u_y^2dyds+C\int_0^t\int_0^1(1+\theta^6)\theta_t^2dyds\\
&\leq
C(T)Y^{\frac{1}{q+5}}+X+XY^{\frac{3}{q+5}}\\
&\leq C(T)\left(1+Z^{\frac{q+8}{q+10}}\right).\label{lemma151}
\end{split}
\end{equation}
and follows from \eqref{freelemma131} and \eqref{density}
\begin{equation}
\begin{split}
&\int_0^t\int_0^1|\mathbf{b}_t|^2dyds\\
&\leq\int_0^t\int_0^1\left(|\mathbf{b}|^2u_y^2+|\mathbf{w}_y|^2+|\mathbf{b}_{yy}|^2+|\mathbf{b}_y|^2v_y^2\right)dyds\\
&\leq
C\int_0^t\int_0^1u_y^2dyds+C\int_0^t\int_0^1(|\mathbf{w}_y|^2+|\mathbf{b}_{yy}|^2)dyds\\
&+C\int_0^t\max_{[0,1]}|\mathbf{b}_y|^2\left(\int_0^1v_y^2dy\right)ds\\
&\leq
C+C\int_0^t\int_0^1\left(|\mathbf{b}_y|^2+|\mathbf{b}_{yy}|^2\right)dyds\\
&\leq C(T).\label{lemma153}
\end{split}
\end{equation}
Substituting \eqref{lemma151} and \eqref{lemma153} into \eqref{lemma150}, it shows that
\begin{equation}
\begin{split}
&\int_0^1u_t^2dy+\int_0^t\int_0^1u_{yt}^2dyds\\
&\leq C\left(1+Z^{\frac{q+8}{q+10}}\right).\label{lemma154}
\end{split}
\end{equation}
On the other hand,  it is also deduced from \eqref{sub22} that
\begin{equation*}
u_{yy}=\frac{v}{\lambda}\left(u_t+\left(p+\frac{|\mathbf{b}|^2}{2}\right)_y
+\frac{\lambda v_yu_y}{v^2}\right)
\end{equation*}
and then by integration, it follows from \eqref{lemma154}
\begin{equation*}
\begin{split}
\int_0^1u_{yy}^2dy&\leq
C\int_0^1\left(u_t^2+p_y^2+|\mathbf{b}|^2\cdot|\mathbf{b}_y|^2+v_y^2u_y^2\right)dy+C\\
&\leq C\int_0^1u_t^2dy+C\int_0^1p_y^2dy+C\int_0^1v_y^2u_y^2dy+C\int_0^1|\mathbf{b}_y|^2dy+C\\
&\leq C\left(1+Z^{\frac{q+8}{q+10}}+\int_0^1(1+\theta^6)\theta_y^2dy
+\left(|\theta^2|^{(0)}+|u_y^2|^{(0)}\right)\int_0^1v_y^2dy\right)\\
&\leq
C\left(1+Z^{\frac{q+8}{q+10}}+Y^{\frac{q+8}{q+5}}+Z^{\frac34}\right)\\
&\leq
C(T)\left(1+Z^{\frac{q+8}{q+10}}\right)\\
\end{split}
\end{equation*}

Thus we have $Z\leq C$ due to $0< \frac{q+8}{q+10}<1$, and $X$
and $Y$ are also bounded. Subsequently, $|\theta|^{(0)}$,
$|u_y|^{(0)}$, $\int_0^1(u_t^2+\theta_y^2+u_{yy}^2)dy$ and
$\int_0^t\int_0^1(u_{yt}^2+\theta_t^2)dyds$ are also bounded.
\end{proof}
In the sequel, the lower bound of temperature is obtained.
\begin{lemma} One has
\begin{equation*}
\theta(y,t)\geq C(T).
\end{equation*}
\end{lemma}
\begin{proof}
Let~$\Theta=\frac{1}{\theta}$, then \eqref{lemma81} is written as
\begin{equation*}
\begin{split}
e_{\theta}\Theta_t&=\left(\frac{\kappa}{v}\Theta_y\right)_y
-\left(\frac{2\kappa\Theta_y^2}{v\Theta}+\frac{\mu|\mathbf{w}_y|^2\Theta^2}{v}+\frac{\nu|\mathbf{b}_y|^2\Theta^2}{v}\right)\\
&~~~~-\frac{\lambda
\Theta^2}{v}\left(u_y-\frac{vp_{\theta}}{2\lambda
\Theta}\right)^2+\frac{vp_{\theta}^2}{4\lambda},
\end{split}
\end{equation*}
which implies
\begin{equation*}
\Theta_t\leq
\frac{1}{e_{\theta}}\left(\frac{\kappa}{v}\Theta_y\right)_y+C(T),
\end{equation*}
by \eqref{rho} and \eqref{lemma17theta} for some positive constant. Define the
operator $\mathscr{L}:=-\frac{\partial}{\partial
t}+\frac{1}{e_{\theta}}\frac{\partial}{\partial y}\left(\frac{\kappa}{v}\frac{\partial}{\partial y}\right)$ and
then
\begin{equation*}
\left\{
\begin{array}{lllllllll}
 \mathscr{L}\widetilde{\Theta}<0,\qquad &\textup{on}\,\,
Q_T=(0,1)\times(0,T),\\[2mm]
\widetilde{\Theta}|_{t=0}\geq 0\qquad &\textup{on}
\,\,[0,1],\\[2mm]
\widetilde{\Theta}_y|_{y=0,1}=0 \quad &\textup{on} \,\, [0,T],
\end{array}
\right.
\end{equation*}
where
$\widetilde{\Theta}(y,t)=C(T)t+\max_{[0,1]}\frac{1}{\theta_0(y)}-\Theta(y,t)$,
and by the comparison theorem, one has
\begin{equation*}
\min_{(y,t)\in \overline{Q_T}}\widetilde{\Theta}(y,t)\geq 0,
\end{equation*}
which is inferred
\begin{equation*}
\theta(y,t)\geq
\left(Ct+\max_{[0,1]}\frac{1}{\theta_0(y)}\right)^{-1},
\end{equation*}
for any $(y,t)\in \overline{Q_T}$.
\end{proof}


\subsection{Higher order derivatives  a priori estimates of $(\mathbf{w},\mathbf{b},\theta)$}
\begin{lemma} One has
\begin{equation}\label{lemma17b}
\begin{split}
&|(\mathbf{w}_y,\mathbf{b}_y,v_y)|^{(0)}
+\int_0^1(|\mathbf{w}_t|^2+|(v\mathbf{b})_t|^2+|\mathbf{w}_{yy}|^2+|\mathbf{b}_{yy}|^2)dy\\
&+\int_0^t\int_0^1\left(|\mathbf{w}_{ty}|^2+|\mathbf{b}_{ty}|^2+\theta_{yy}^2\right)dyds\leq
C(T).
\end{split}
\end{equation}
\end{lemma}
\begin{proof}
Firstly, we differentiate \eqref{sub33} with respect to $t$,
multiply it with $\mathbf{w}_t$ and then integrate over
$(0,1)\times (0,t)$
\begin{equation*}
\begin{split}
&\frac12\int_0^1|\mathbf{w}_t|^2dy+\int_0^t\int_0^1\frac{\mu}{v}|\mathbf{w}_{ty}|^2dyds\\
&=\frac12\int_0^1|\mathbf{w}_t|^2(y,0)dy
+\int_0^t\int_0^1\frac{\mu}{v^2}u_y\mathbf{w}_y\cdot\mathbf{w}_{ty}dyds
-\int_0^t\int_0^1\mathbf{b}_t\cdot\mathbf{w}_{ty}dyds\\
&\leq
C+\frac12\int_0^t\int_0^1\frac{\mu}{v}|\mathbf{w}_{ty}|^2dyds+C(T)
\int_0^t\int_0^1(u_y^4+|\mathbf{w}_y|^4+|\mathbf{b}_t|^2)dyds,
\end{split}
\end{equation*}
which follows from \eqref{freelemma131}, \eqref{lemma17theta} and \eqref{lemma17}
\begin{equation}\label{lemma17w}
\int_0^1|\mathbf{w}_t|^2dy+\int_0^t\int_0^1|\mathbf{w}_{ty}|^2dyds\leq
C(T).
\end{equation}
Similarly, we get from \eqref{sub33}
\begin{equation*}
\mathbf{w}_{yy}=\frac{v}{\mu}\left(\mathbf{w}_t-\mathbf{b}_y
+\frac{\mu}{v^2}v_y\mathbf{w}_y\right),
\end{equation*}
which leads to
\begin{equation*}
\begin{split}
\int_0^1|\mathbf{w}_{yy}|^2dy&\leq
C\int_0^1(|\mathbf{w}_t|^2+|\mathbf{b}_y|^2+v_y^2|\mathbf{w}_y|^2)dy\\
&\leq C+C\max_{[0,1]}|\mathbf{w}_y|^2\int_0^1v_y^2dy\\
&\leq
C(T)+C\int_0^1|\mathbf{w}_y|^2dy+\frac12\int_0^1|\mathbf{w}_{yy}|^2dy\\
&\leq
C(T)+\frac12\int_0^1|\mathbf{w}_{yy}|^2dy.
\end{split}
\end{equation*}
i.e.,
\begin{equation*}
\int_0^1|\mathbf{w}_{yy}|^2dy\leq
C(T).
\end{equation*}
Secondly, we further deduce that
\begin{equation}
\begin{split}
&\int_0^t\int_0^1\frac{\kappa^2}{v^2}\theta_{yy}^2dyds\\
 &\leq
C\int_0^t\int_0^1(\theta_t^2+u_y^2+u_y^4+|\mathbf{w}_y|^4+|\mathbf{b}_y|^4+v_y^2\theta_y^2+\theta_y^4)dyds\\
&\leq C+C\int_0^t\max_{[0,1]}\theta_y^2\int_0^1(v_y^2+\theta_y^2)dyds\\
&\leq
C+C\int_0^t\int_0^1\theta_y^2dyds+\frac12\int_0^t\int_0^1\frac{\kappa^2}{v^2}\theta_{yy}^2dyds,
\notag
\end{split}
\end{equation}
and then
\begin{equation}
\int_0^t\int_0^1\theta_{yy}^2dyds\leq
C(T)+C\int_0^t\int_0^1\theta_y^2dyds\leq  C(T).
\end{equation}
In addition, by \eqref{lemma2V}, one has
\begin{equation*}
\begin{split}
\lambda (\ln v)_y=&\lambda (\ln v_0(y))_y+\int_0^t(p_y+\mathbf{b}\cdot\mathbf{b}_y)ds+(u-u_0)
\end{split}
\end{equation*}
or equivalently
\begin{equation*}
\begin{split}
v_y^2
&\leq
C+C\int_0^t(|\mathbf{b}_y|^2+p_v^2v_y^2+p_{\theta}^2\theta_y^2)ds\\
&\leq
C+C\int_0^t\int_0^1(|\mathbf{b}_y|^2+|\mathbf{b}_{yy}|^2+\theta_y^2
+\theta_{yy}^2)dyds+C\int_0^tv_y^2ds\\
&\leq C+C\int_0^tv_y^2ds.
\end{split}
\end{equation*}
Thus, we deduce that $v_y$ is bounded by Gr\"{o}nwall inequality.
Lastly, differentiate \eqref{sub44} with respect to $t$ and then
multiply by $(v\textbf{b})_t$ and integrate to get
\begin{equation}
\begin{split}
&\frac12\frac{d}{dt}\int_0^1(v\mathbf{b})_t^2dy+\int_0^1\nu|\mathbf{b}_{yt}|^2dy\\
= &
-\int_0^1\frac{\nu}{v}\mathbf{b}_{yt}\cdot(\mathbf{b}u_{yy}+\mathbf{b}_yu_y+\mathbf{b}_tv_y)dy\\
&+\int_0^1\frac{\nu}{v^2}u_y\mathbf{b}_y\cdot(\mathbf{b}u_{yy}+\mathbf{b}_yu_y
+\mathbf{b}_tv_y+\mathbf{b}_{ty}v)dy\\
&-\int_0^1\mathbf{w}_t\cdot(\mathbf{b}u_{yy}+\mathbf{b}_yu_y+\mathbf{b}_tv_y+\mathbf{b}_{ty}v)dy\\
\le &\varepsilon
\int_0^1\nu|\mathbf{b}_{yt}|^2dy+C\int_0^1(\mathbf{b}_y^2+\mathbf{b}_t^2+\mathbf{w}_t^2)dy+C\int_0^1(u_{yy}^2+|\mathbf{w}_{yy}|^2)dy.
\end{split}
\end{equation}
So we get from \eqref{freelemma131}, \eqref{lemma17} and \eqref{lemma17w}
\begin{equation}
\int_0^1(v\mathbf{b})_t^2dy+\int_0^t\int_0^1|\mathbf{b}_{yt}|^2dyds\leq
C,
\end{equation}
by integration for any fixed sufficiently small $\varepsilon$.
Furthermore, by \eqref{sub44}, one has
\begin{equation*}
\begin{split}
\int_0^1|\mathbf{b}_{yy}|^2dy&\leq
C\int_0^1((v\mathbf{b})_t^2+|\mathbf{w}_y|^2+v_y^2|\mathbf{b}_y|^2)dy\\
&\leq C(T)+C\int_0^1|\mathbf{b}_y|^2dy\leq C(T),
\end{split}
\end{equation*}
and similarly
\begin{equation*}
|\mathbf{b}_y|^2\leq
C\int_0^1|\mathbf{b}_y|^2dy+C\int_0^1|\mathbf{b}_{yy}|^2dy\leq C(T).
\end{equation*}
This completes the proof.
\end{proof}

\begin{lemma}\label{Lemma18}We have
\begin{equation}
|\theta_y|^{(0)}+\int_0^1(\theta_t^2+\theta_{yy}^2)dy+\int_0^t\int_0^1\theta_{yt}^2dyds\leq
C(T).\label{lemma181}
\end{equation}
\end{lemma}
\begin{proof}
Differentiate \eqref{sub55} with respect to $t$, multiply it by
$e_{\theta}\theta_t$ and then integrate it over $[0,1]$
\begin{equation*}
\begin{split}
\frac{d}{dt}&\int_0^1\frac{(e_{\theta}\theta_t)^2}{2}dy+\int_0^1\frac{\kappa}{v}e_{\theta}\theta_{yt}^2dy\\
=&\int_0^1\bigg[-p_{\theta}e_{\theta}u_y\theta_t^2-\theta p_{\theta
v}e_{\theta}u_y^2\theta_t-\theta
p_{\theta\theta}e_{\theta}u_y\theta_t^2-\theta
p_{\theta}e_{\theta}u_{yt}\theta_t\\
&+\frac{2\lambda}{v}e_{\theta}u_yu_{yt}\theta_t-\frac{\lambda}{v^2}e_{\theta}u_y^3\theta_t
+\frac{2\mu}{v}e_{\theta}\mathbf{w}_y\cdot\mathbf{w}_{yt}\theta_t-\frac{\mu}{v^2}e_{\theta}u_y|\mathbf{w}_y|^2\theta_t\\
&+\frac{2\nu}{v}e_{\theta}\theta_t\mathbf{b}_y\cdot\mathbf{b}_{yt}-\frac{\nu}{v^2}e_{\theta}|\mathbf{b}_y|^2\theta_t
-\left(\frac{\kappa}{v}\right)_ve_{\theta v}v_yu_y\theta_y
\theta_t-\left(\frac{\kappa}{v}\right)_ve_{\theta\theta}u_y\theta_y^2\theta_t\\
&-\left(\frac{\kappa}{v}\right)_ve_{\theta}u_y\theta_y\theta_{yt}-\left(\frac{\kappa}{v}\right)_{\theta}e_{\theta
v}v_y\theta_y\theta_t^2
-\left(\frac{\kappa}{v}\right)_{\theta}e_{\theta\theta}\theta_y^2\theta_t^2\\
&-\left(\frac{\kappa}{v}\right)_{\theta}e_{\theta}\theta_t\theta_y\theta_{yt}
-\frac{\kappa}{v}e_{\theta v}v_y
\theta_t\theta_{yt}-\frac{\kappa}{v}e_{\theta\theta}\theta_y\theta_t\theta_{yt}\bigg]dy,
\end{split}
\end{equation*}
which implies by integration
\begin{equation*}
\begin{split}
&\int_0^1\theta_t^2dy+\int_0^t\int_0^1\theta_{yt}^2dyds\\
&\leq
C\left(1+\int_0^t\int_0^1(v_y^2+\theta_y^2)\theta_t^2dyds\right)\\
&\leq C\left(1+\int_0^t\max_{[0,1]}\theta_t^2ds\right)\\
&\leq C\left(1+\varepsilon
\int_0^t\int_0^1\theta_{yt}^2dyds+C_{\varepsilon}\int_0^t\int_0^1\theta_t^2dyds\right),
\end{split}
\end{equation*}
and then $\int_0^1\theta_t^2dy$ and
$\int_0^t\int_0^1\theta_{yt}^2dyds$ are bounded. Taking the same
operation as $\mathbf{b}_{yy}$ and $\mathbf{w}_{yy}$, we deduce from
\eqref{sub55}
\begin{equation*}
\begin{split}
\int_0^1\theta_{yy}^2dy&\leq
\int_0^1\frac{1}{\kappa^2}\left(1+\theta_t^2+u_y^2+u_y^2+v_y^2\theta_y^2+\theta_y^4\right)dy\\
&\leq
C\left(1+\max_{[0,1]}\theta_y^2\cdot\int_0^1(v_y^2+\theta_y^2)dy\right)\\
&\leq \varepsilon \int_0^1\theta_{yy}^2dy+C(T),
\end{split}
\end{equation*}
which implies the boundedness of $\int_0^1\theta_{yy}^2dy$, and so
is $|\theta_y|^{(0)}$ by employing the embedding theorem.
\end{proof}


Now  we have established Lemmas \ref{Lemma1}--\ref{Lemma18}, we
can prove that the H\"{o}lder estimates of solutions by the routine manner. Indeed, from
\eqref{lemma17theta}, \eqref{lemma17b} and \eqref{lemma181}, one has
\begin{equation*}
(u,\mathbf{w},\mathbf{b},\theta)\in C^{1,0}(Q_T)^6.
\end{equation*}
Following the procedure of \cite{DafermosHisao}, \cite{umeharaTani}
or \cite{ZhangXie}, we can get
\begin{equation*}
\begin{split}
&(u,\mathbf{w},\mathbf{b},\theta)\in C^{1,\frac12}(Q_T)^6, \\[3mm]
 &(u_y,\mathbf{w}_y,\mathbf{b}_y,\theta_y)\in
C^{\frac13,\frac16}(Q_T)^6.
\end{split}
\end{equation*}
Finally, by the classical schauder estimates and the methods in
\cite{Lady}, the H\"{o}lder estimates of solutions are derived. This
completes the proof of Theorem \ref{thm}.
\section*{Acknowledgements} This work is partially supported by
National Nature Science Foundation of China (NNSFC-10971234,
NNSFC-10962137,NNSFC-11001278), China Postdoctoral Science
Foundation funded project (NO. 20090450191), Natural Science
Foundation of GuangDong (NO. 9451027501002564) and the Fundamental
Research Funds for the Central Universities.


\end{document}